\documentclass[a4paper,12pt]{amsart}
\usepackage{preamble}
\usepackage{xurl}
\title{Compact Lie groups isolated up to conjugacy}
\author[B. Csikós]{Balázs Csikós}
\address{Eötvös Loránd University, Institute of Mathematics, Budapest, Pázmány P. stny. 1/C.}
\email{balazs.csikos@ttk.elte.hu}
\author[T. Kátay]{Tamás Kátay}
\address{Eötvös Loránd University, Institute of Mathematics, Budapest, Pázmány P. stny. 1/C.}
\email{13heted@gmail.com}
\author[A. Kocsis]{Anett Kocsis}
\address{Eötvös Loránd University, Institute of Mathematics, Budapest, Pázmány P. stny. 1/C.}
\email{sakkboszi@gmail.com}
\author[M.~Pálfy]{Máté Pálfy}
\address{Eötvös Loránd University, Institute of Mathematics, Budapest, Pázmány P. stny. 1/C.}
\email{palfymateandras@gmail.com}

\begin{document}
	\subjclass[2020]{22A05, 22E15, 54B20}
	\keywords{Vietoris topology, space of compact subgroups}
	\maketitle
	
	\begin{abstract}
		The set $\cals(G)$ of compact subgroups of a Hausdorff topological group $G$ can be equipped with the Vietoris topology. A compact subgroup $K\in\cals(G)$ is \emph{isolated up to conjugacy} if there is a neighborhood $\calu\subseteq\cals(G)$ of $K$ such that every $L\in\calu$ is conjugate to $K$. In this paper, we characterize compact subgroups of a Lie group that are isolated up to conjugacy. Our characterization depends only on the  \emph{intrinsic} structure of $K$,  the ambient Lie group $G$ and the embedding of $K$ into $G$ are irrelevant. 
		In addition, we prove that any continuous homomorphism from a compact group $G$ onto a compact Lie group $H$ induces a continuous open map from $\cals(G)$ onto $\cals(H)$.
	\end{abstract}
	
	\section{Introduction}
	
	Every topological space in this paper is assumed to be Hausdorff.
	
	For any topological space $X$, one may equip the set $\calk(X)$ of non-empty compact subsets of $X$ with the so-called Vietoris topology (for the definition see Subsection~\ref{ss.the_hyperspace}). The space $\calk(X)$ is often called the \emph{hyperspace of compact subsets of $X$}. In fact, $\calk$ is a functor from the category of topological spaces into itself, which assigns to a continuous map $f\colon X\to Y$ the continuous map $\calk(f)\colon \calk(X)\to \calk(Y)$ defined by $\calk(f)\colon  A\mapsto f(A)$ for $A\in \calk(X)$. 
	
	For a topological group $G$, the set $\cals(G)$ of compact subgroups of $G$ is a closed subspace of $\calk(G)$. We call $\cals(G)$ \emph{the (hyper)space of compact subgroups}. Setting $\cals(f)=\calk(f)$ for any continuous homomorphism $f$ between two topological groups, $\cals$ becomes a functor from the category of topological groups into the category of topological spaces. S.~Fisher and P.~Gartside \cite{fisher_gartside_1}, \cite{fisher_gartside_2} studied basic properties of $\cals(G)$ in the case of a compact group $G$, with an emphasis on profinite groups. In this paper, we deal with other topological aspects of the functor $\cals$.
	\begin{remark}
		Fell \cite{Fell} defined a variant of the Vietoris topology on the family of closed subsets of a locally compact space. Fell's topology can also be used to equip the space of closed (or compact) subgroups of a locally compact group with a very interesting and useful topology. Fell's topology on the set of closed subgroups of a locally compact topological group was defined by Chabauty \cite{Chabauty} much before Fell's work in order to extend Mahler's compactness theorem on lattices in a Euclidean space to lattices in a locally compact topological group. The Chabauty--Fell topology and the Vietoris topology on closed subgroups of a topological group coincide when the ambient group is compact, but there are crucial differences between them in the general case, so they should not be confused. 
	\end{remark} 
	
	It is natural to ask which compact subgroups of a given topological group $G$ can be approximated by other compact subgroups (with respect to the Vietoris topology)? That is, what are the limit points of $\cals(G)$? For a compact connected group $G$, every non-normal compact subgroup $K\leq G$ is a limit point since $K$ can be approximated by its conjugates (for any neighborhood $U$ of the identity of $G$ pick any $g\in U\setminus N_G(K)$ and take $gKg^{-1}$). Therefore, one may want to refine the question to exclude this trivial possibility. We say that a compact subgroup $K$ of a topological group $G$ is \emph{isolated up to conjugacy (in $G$)} if there is a neighborhood $\calu\subseteq\cals(G)$ of $K$ such that every $L\in\calu$ is conjugate to $K$. The refined question is the following:
	
	\begin{question}
		Which compact subgroups of a topological group $G$ are isolated up to conjugacy?
	\end{question}
	
	A variant of this problem was studied by A.~M.~Turing \cite{Turing}. He proved that if a connected Lie group can be approximated by finite subgroups, then it is compact and abelian.\footnote{We warn the reader that Turing formulates his theorem without the assumption of connectedness, but this condition cannot be omitted as the example of the direct product of a torus and a non-commutative finite group shows. Actually, the proof of Turing works (only) if the group $G$ is connected.} Another variant was studied by D.~Montgomery and L.~Zippin \cite{Montgomery_Zippin}. Although they did not formulated their theorem as a result on the space of compact subgroups, it implies that whenever a compact subgroup $K$ of a Lie group $G$ is not isolated up to conjugacy, $K$ can be approximated by its proper subgroups. 
	
	Building on the works of Turing, Montgomery and Zippin we obtain the following result (Theorem \ref{t.isolated_characterized}), which answers also Question~6.6. in \cite{generic_top_groups}.
	
	\begin{theorem*}
		Let $G$ be a Lie group. For a compact subgroup $K\leq G$ the following are equivalent:
		
		(1) The identity component $K^{\circ}$ of $K$ is not perfect, i.e., $K^\circ\neq (K^\circ)'$.
		
		(2) The subgroup $K$ has a quotient of the form
		$$\circle^m\rtimes F,$$
		where $m>0$,  $\circle^m$ is an $m$-dimensional torus and $F$ is a finite group.
		
		(3) There is a sequence of closed proper subgroups $K_n\lneqq K$ such that $K_n{\to}K$.
		
		(4) The point $K\in\cals(G)$ is not isolated up to conjugacy.
		
		(5) The conjugacy  set $\{gKg^{-1}\mid g\in G\}$ of $K$ is not open in $\cals(G)$.
	\end{theorem*}
	
	The above theorem fails to be true for arbitrary topological groups $G$. Indeed, the trivial subgroup $\{e_G\}$ of $G$ never satisfies condition (1), but it satisfies condition (4) if $G$ does not have the NSS (no small subgroup) property.
	
	As condition (1) depends only on the intrinsic structure of $K$, the theorem also implies that a compact Lie group $K$ is isolated up to conjugacy in an arbibtrary Lie group $G$ if and only if it is isolated up to conjugacy in one single Lie group $G\ge K$, for example in itself. 
	
	We can define a quasi-order $\succeq$ on the family of  compact Lie groups by saying $G\succeq H$ for the compact Lie groups $G$ and $H$ if and only if there is a continuous epimorphism $G\twoheadrightarrow H$. As usual, calling two compact Lie groups $G$ and $H$ \emph{quasi-isomorphic} if both $G\succeq H$ and $H\succeq G$, the quasi-order $\succeq$ induces a partial order on the quasi-isomorphism classes. The equivalence  (2)$\iff$(4) shows that the quasi-isomorphism classes of compact Lie groups not isolated up to conjugacy in a Lie group form an upper set in this poset, and property (2) allows us to describe minimal elements in this set. These minimal elements are represented by groups of the form $\circle^m\rtimes_{\alpha} F$, where $m>0$, and $F$ is a finite group with a faithful and rationally irreducible conjugacy representation $\alpha\colon F\to\textrm{GL}(m,\mathbb Z)\cong \textrm{Aut}(\circle^m)$. (See Proposion \ref{prop:minimal_groups}.)
	
	In Section~\ref{s.an_open_map}, we turn to another problem concerning the space of compact subgroups. We shall show that if $X$, $Y$ are topological spaces, $X$ is locally compact and $f\colon X\to Y$ is a surjective continuous open map, then the induced map $\calk(f)\colon\calk(X)\to\calk(Y)$ is also a surjective continuous open map. (See Proposition \ref{p.continuous_image} below.)  
	
	Since any surjective continuous homomorphism between compact groups is open, one encounters the above situation quite often in the study of compact groups. Thus the following problem arises:
	
	\begin{question}\label{q.open}
		For which compact groups $G$ and $H$ is it true that for any surjective continuous homomorphism $f\colon G\to H$ the map $\cals(f)\colon\cals(G)\to\cals(H)$ is open?
	\end{question}
	
	As a further application of the above quoted theorem of Montgomery and Zippin, we prove that this holds if $G$ is a compact topological group and $H$ is a compact Lie group. (See Theorem \ref{t.open_map}.)
	
	\section{Preliminaries}
	
	\subsection{The hyperspace of compact subsets}\label{ss.the_hyperspace}
	
	For a topological space $X$ let $\calk(X)$ denote the set of non-empty compact subsets of $X$. It is a topological space with the Vietoris topology, that is, the topology generated by basis elements of the form
	\begin{equation}\tag{$\star$}
		\calv(U_0,\dots,U_n)=\{K\in \calk(X):\ K\subseteq U_0\text{ and }\ K\cap U_i\neq\emptyset\text{ for }1\leq i\leq n\}
	\end{equation}
	with $n\in\nat$ and $U_0,\ldots,U_n$ open in $X$.
	
	If $(X,d)$ is a metric space, then $\calk(X)$ is a metric space with the Hausdorff metric
	$$d_H(K,L)\defeq \inf\{\eps>0:\ K_\eps\supseteq L,\ L_\eps\supseteq K\},$$
	where $A_\eps$ is the open $\eps$-neigborhood of the set $A$. It is well-known that the Hausdorff metric induces the Vietoris topology. It is also well-known that $\calk(X)$ inherits several topological properties of $X$. For example, if $X$ is compact, then $\calk(X)$ is also compact. For details, see \cite[Sec.~4.F]{KECHRIS} and \cite[Thm.~4.2]{MICHAEL}.
	
	The following lemma collects some useful facts. 
	
	\begin{lemma}\label{lem:_folytonossag}\mbox{ }
		
		(1) For any continuous map $f\colon X\to Y$, the induced map $\calk(f)\colon \calk(X)\to \calk(Y)$ is continuous.
		
		(2) For any topological space $X$, the map  $\iota_X\colon X\to \calk(X)$, $\iota_X(x)=\{x\}$ is a topological embedding.  
		
		(3) For any family of topological spaces $\langle X_i:i\in I\rangle$, the map $P\colon \bigtimes_{i\in I}\calk(X_i)\to \calk\left(\bigtimes_{i\in I}X_i\right)$, $\langle A_i\in\calk(X_i):i\in I\rangle\mapsto \bigtimes_{i\in I}A_i$ is continuous.
	\end{lemma}
	\begin{remark}
		Statement (3) is claimed without proof in \cite[Lem. 1]{fisher_gartside_1} for a pair of spaces, and left as an exercise for a pair of metric spaces in \cite[Ex. 4.29.vii]{KECHRIS}. We include the proof for the convenience of the reader.
	\end{remark}
	
	\begin{proof}
		(1) is a consequence of the identity \[\calk(f)^{-1}\big(\calv(U_0,\dots,U_n)\big)=\calv\big(f^{-1}(U_0),\dots,f^{-1}(U_n)\big)\quad \forall\, U_0,\dots,U_n\text{ open in }Y,\]
		where $\calv({\dots})$ is defined by equation $(\star)$. 
		
		(2) follows from the identity 
		\[\iota_X^{-1}\big(\calv(U_0,\dots,U_n)\big)=\bigcap_{i=0}^n U_i\quad \forall\, U_0,\dots,U_n\text{ open in }X.\]
		
		To prove (3), take an arbitrary element $\langle A_i:i\in I\rangle\in \bigtimes_{i\in I}\calk(X_i)$ and an open neighborhood $\calv(U_0,\dots,U_n)$ of its image $A:=\bigtimes_{i\in I}A_i$. Choose for all $a=\langle a_i:i\in I\rangle\in  A$ an open neighborhood $W_a$ of $a$ such that \begin{itemize}
			\item $W_a$ has the form $W_a=\bigtimes_{i\in I}W_a^i$\,, where $W_a^i$ is an open subset of $X_i$ and $W_a^i=X_i$ for all but a finite number of $i\in I$;
			\item $W_a\subseteq U_j$ whenever $a\in U_j$, $0\leq j\leq n$.
		\end{itemize}
		Since $A$ is compact, one can find a finite number of points $b_1,\dots b_N\in A$ such that $ A\subseteq \bigcup_{k=1}^N W_{b_k}$. For $i'\in I$, let $\pi_{i'}\colon \bigtimes_{i\in I}X_i \to X_{i'}$ be the canonical projection and define the open subset $V_0^i\subseteq X_i$ by
		\[
		V_0^i=\bigcup_{a_i\in A_i}\Big(\bigcap_{\pi_i^{-1}(a_i)\cap W_{b_k}\neq\emptyset} W_{b_k}^i\Big).
		\]
		
		For $1\leq j\leq n$, choose an element $a_j\in  A\cap U_j$ and set $V_j^i=W_{a_j}^i$.
		
		It is straightforward to check that $\bigtimes_{i\in I}\calv(V_0^i,\dots,V_n^i)$ is an open neighborhood of 
		${\langle A_i:i\in I\rangle}$ which is mapped by $P$ into $\calv(U_0,\dots,U_n)$.
	\end{proof}
	
	\begin{prop}\label{p.continuous_image}
		Let $X$, $Y$ be topological spaces, and assume that $X$ is locally compact. Then for any continuous surjective open map $f\colon X\to Y$, the map $\calk(f)\colon \calk(X)\to\calk(Y)$ is also surjective and open.
	\end{prop}
	
	\begin{proof}
		
		First we prove that $\calk(f)$ is surjective. Pick any $L\in \calk(Y)$. Let $C_x$ be a compact neighborhood of $x$ for every $x\in X$. Then $\bigcup_{x\in X} f(\interior C_x)$ is an open cover of $Y$, hence there exist finitely many $C_{x_1},\ldots,C_{x_n}$ such that the image of $C=\bigcup_{i=1}^n C_{x_i}$ covers $L$. Thus $K\defeq f|_C^{-1}(L)$ is a compact set that maps onto $L$.
		
		To show that $\calk(f)$ is open, it is enough to check that for any choice of the the open sets $U_0,\dots,U_n\subseteq X$, the $\calk(f)$ image of  $\calv(U_0,\dots,U_n)$ defined by $(\star)$ is open in $\calk(Y)$. Take an arbitrary element $L\in \calk(f)\big(\calv(U_0,\dots,U_n)\big)$ and choose $K\in \calv(U_0,\dots,U_n)$ such that $f(K)=L$. Since $X$ is locally compact, we can find $C\in\calk(X)$ such that $K\subseteq\interior C\subseteq C\subseteq U_0$. Put $V_i\defeq f(U_i\cap \interior C)$ for $0\leq i\leq n$. Observe that the sets $V_i$ are open in $Y$ and $\calv(V_0,\dots,V_n)$ is an open neighborhood of $L$.
		
		Furthermore, for any $\hat L\in \calv(V_0,\dots,V_n)$,  the set $\hat K\defeq (f|_C)^{-1}(\hat L)$ is in $\calv(U_0,\dots,U_n)$ and $f$ maps $\hat K$ onto $\hat L$, therefore $\calv(V_0,\dots,V_n)\subseteq \calk(f)\big(\calv(U_0,\dots,U_n)\big)$ and this completes the proof.
	\end{proof}
	
	\subsection{The hyperspace of compact subgroups}
	
	As the group operations of a topological group $G$ are continuous, Lemma \ref{lem:_folytonossag} implies the continuity of some frequently used actions on $\calk(G)$.
	
	\begin{lemma}\label{l.right_multip_cont}
		For a topological group $G$, the actions of $G$ on $\calk(G)$ by left translations, right translations, and conjugations given by the maps $L\colon G\times \calk(G)\to \calk(G)$, $L(g,K)=gK$; $R\colon  \calk(G) \times G \to \calk(G)$, $R(K,g)=Kg$; and $C\colon G\times \calk(G)\to \calk(G)$, $C(g,K)=gKg^{-1}$ respectively, are continuous. 
	\end{lemma}
	
	\begin{prop}
		For any topological group $G$ the set
		$$\cals(G)\defeq\{K\in \calk(G):\ K\text{ is a subgroup of }G\}$$
		is closed in $\calk(G)$.
	\end{prop}
	
	\begin{proof}
		Observe that
		$$\cals(G)=\{K\in\calk(G):\ K\cdot K^{-1}=K\}.$$
		The map $K\mapsto K\cdot K^{-1}$ equals the composition 
		\[\calk(G)\stackrel{\Delta}{\longrightarrow}\calk(G)\times \calk(G)\stackrel{P}{\longrightarrow} \calk(G\times G)\stackrel{\calk(D)}{\longrightarrow}\calk(G),\] 
		where $\Delta$ is the diagonal map, $P$ is defined as in Lemma \ref{lem:_folytonossag} (3) and $D\colon G\times G\to G$ is the division map $D(a,b)=ab^{-1}$, hence it is is continuous by Lemma \ref{lem:_folytonossag}, and its  fixed point set $\cals(G)$ is closed.
	\end{proof}
	
	\begin{defi}
		We call the above defined $\cals(G)$ \emph{the space of compact subgroups} of $G$.
	\end{defi}
	Note that for a compact metrizable group $G$, the space $\cals(G)$ is also compact metrizable.
	
	We will need the following well-known fact. 
	
	\begin{prop}[{\cite[Thm. 1.6]{OPEN_MAP}}]\label{p.quotient_map}
		Let $G,H$ be compact groups. If $\varphi\colon G\to H$ is a surjective continuous homomorphism, then it is a quotient map (of topological groups). In particular, it is open.
	\end{prop}
	
	\begin{notation}\label{n.identity_component}For a topological group $G$ let $G^{\circ}$ denote the connected component of the identity of $G$ (also referred to as the identity component), which is a closed normal subgroup of $G$.
	\end{notation}
	
	\subsection{Lie groups}
	
	We recall some well-known facts about Lie groups.
	
	\begin{remark}
		Following Helgason \cite{Helgason}, we define a Lie subgroup of a Lie group $G$ as the image $\varphi(H)$ of a Lie group $H$ under a smooth injective immersion $\varphi\colon H\to G$ that is a group homomorphism as well.
	\end{remark}
	
	\begin{theorem}[See \cite{Serre}]
		\label{t.lie_group_fun_facts}\mbox{}
		
		(1) Let $H\leq G$ be Lie groups. If $\dim H=\dim G$, then $H^\circ=G^\circ$.
		
		(2) A quotient of a Lie group by a closed subgroup is a homogeneous $G$-manifold. In the special case of a closed normal subgroup, the quotient space is a Lie group.
		
		(3) For every Lie algebra $\frh$ there exists a unique connected, simply connected Lie group $H$ with Lie algebra $\frh$.
		
		Moreover, if $G$ is a Lie group with Lie algebra $\frg$ and $\psi\colon \frh\to\frg$ is a Lie algebra homomorphism, then there exists a unique smooth homomorphism $\varphi$ that makes the following diagram commute:
		
		\[\begin{tikzcd}
			{H} && G \\
			\\
			{\mathfrak{h}} && {\mathfrak{g}}
			\arrow["{\varphi}", from=1-1, to=1-3]
			\arrow["{\exp_{H}}", from=3-1, to=1-1]
			\arrow["\psi", hook, from=3-1, to=3-3]
			\arrow["{\exp_G}", from=3-3, to=1-3]
		\end{tikzcd}\]
		where $\exp$ is the exponential map. If $\psi$ is injective, then $\varphi$ is an immersion.
		
		(4) Any closed subgroup $H$ of a Lie group $G$ is a Lie subgroup of $G$. (Cartan's Theorem)
	\end{theorem}
	
	The following theorem describes a correspondence between normal Lie subgroups of a Lie group $G$ and ideals of its Lie algebra $\frg$, see \cite[Ch. II. p.128]{Helgason}.
	
	\begin{theorem}\label{t.normal_ideal}
		Let $G$ be a connected Lie group with Lie algebra $\frg$. Then $H\mapsto\frh$ is a bijection between the connected normal Lie subgroups of $G$ and the ideals of $\frg$. In particular, if the Lie algebra $\frh$ of a connected Lie subgroup $H\leq G$ is an ideal in $\frg$, then $H\nsub G$.
	\end{theorem}
	
	Recall that a Lie algebra $\frg$ is called compact if it is the Lie algebra of a compact Lie group (cf. \cite[Ch. II. Cor. 6.7]{Helgason}). The following statements are well-known.
	
	\begin{theorem}\label{t.lie_alg_decomposition} \mbox{}
		
		(1) The Lie algebra $\frg$ is compact  if and only if $\frg$ can be equipped with a positive definite $\ad_{\frg}$-invariant symmetric bilinear function. In particular, every subalgebra of a compact Lie algebra is compact.
		
		(2) Any compact Lie algebra $\frg$ is reductive, hence it can be decomposed as $\frg=[\frg,\frg]\oplus \real^l$, where $[\frg,\frg]$ is compact and semisimple, the summand $\real^l$ is the center of $\frg$, see \cite[Ch. II. Prop. 6.6 (ii)]{Helgason}. Moreover, if $G$ is a compact connected Lie group with Lie algebra $\frg$, then the commutator subgroup $G'$ of $G$ is a compact connected Lie subgroup of $G$ with Lie algebra $[\frg,\frg]$, see \cite[Cor. 6.31 (iii)]{Hofmann_Morris}.
		
		(3) Any semisimple compact Lie algebra $\frg$ can be decomposed as the direct sum $\frg_1\oplus\ldots\oplus\frg_k$ of compact simple Lie algebras $\frg_i$. Any ideal of $\frg$ is the direct sum of some of the ideals $\frg_i$, see \cite[Ch. II. Cor. 6.3]{Helgason}
		
		(4) Any connected Lie group with a compact semisimple Lie algebra is compact. 
	\end{theorem}
	The last statement (4) is a corollary of the following theorem. 
	
	\begin{theorem}\label{t.exp_surj}
		Let $\frg$ be a compact semisimple Lie algebra with a fixed $\ad_{\frg}$-invariant inner product $\langle\,,\rangle$, $\varrho$ be the corresponding Euclidean metric on $\frg$. Then there exists $\Delta>0$ such that $\exp_G(B_\varrho(0,\Delta))=G^{\circ}$ for any group $G$ with Lie algebra $\frg$.
	\end{theorem}
	
	\begin{proof}
		We may assume without loss of generality that $G$ is connected. It is known (\cite[Prop. 3.39]{cheeger_ebin}) that $\langle\,,\rangle$ extends to a bi-invariant Riemannian metric on $G$, which turns $G$ into a Riemannian symmetric space, the curvature tensor of which satisfies the identity
		\[
		R(x,y)x=-\frac{1}{4}\ad(x)^2(y) \text{ for all }x,y\in \frg.
		\]
		If $x\in\frg$ is a unit vector and we extend it to an orthonormal basis $x=e_1,\dots,e_n$ of $\frg$, then the Ricci curvature of $G$ in the direction $x$ is expressed by
		\[
		Ric(x,x)=\sum_{i=1}^n\langle R(x,e_i)x,e_i\rangle=\frac{1}{4}\sum_{i=1}^n\|\ad(x)(e_i)\|^2=\frac{1}{4}\|\ad(x)\|_{HS}^2,
		\]
		where $\|.\|_{HS}$ stands for the Hilbert-Schmidt norm. As $\frg$ has trivial center, the Ricci curvature has a strictly positive minimum
		\[
		Ric_{\min}=\frac{1}{4}\min_{\|x\|=1}\|\ad(x)\|_{HS}^2>0.
		\]
		Thus, by Myers' theorem \cite[Thm. 1.26]{cheeger_ebin}, any two points of $G$ can be connected by a geodesic curve of length at most 
		\[
		\Delta=\pi\sqrt{\frac{n-1}{Ric_{\min}}}.
		\]
		As geodesic curves of $G$ starting from the identity element of $G$ coincide with one-parameter subgroups of $G$, the above $\Delta$ satisfies the condition. 
	\end{proof}
	
	The following lemma is simple.  We present a proof for completeness.
	\begin{lemma}\label{l.nonconjugate}
		Let $K$ be a compact subgroup of a Lie group $G$. Then $K$ is not isomorphic to any of its proper subgroups.
	\end{lemma}
	
	\begin{proof}
		It follows easily from Theorem~\ref{t.lie_group_fun_facts} (1) that there is no infinite strictly decreasing sequence $K_1\gneqq K_2\gneqq\ldots$ of compact Lie groups. If there were an isomorphism $\varphi\colon K\to \varphi(K)\lneqq K$, then $K\gneqq \varphi(K)\gneqq \varphi^2(K)\gneqq \varphi^3(K)\gneqq\ldots$ would be such an infinite strictly decreasing sequence.
	\end{proof}

	We will need the following theorem of Montgomery and Zippin \cite[Thm.  1]{Montgomery_Zippin}.
	
	\begin{theorem}\label{t.mz}
		Let $G$ be a Lie group and let $K$ be a compact subgroup of $G$. Then there exists an open set $U\supseteq K$ such that for every subgroup $H$ of $G$ that lies in $U$, there exists $g\in G$ such that $gHg^{-1}\leq K$.
		
		Moreover, for any neighborhood $O$ of the identity of $G$, there is an open set $U_O\supseteq K$ such that for every subgroup $H$ of $G$ that lies in $U_O$ there exists $g\in O$ such that $gHg^{-1}\leq K$.
	\end{theorem}
	
	The second assertion is stated in \cite[Cor. on p. 451]{Montgomery_Zippin}.
	The following is an immediate corollary of a theorem of A.~M.~Turing, \cite[Thm. 2]{Turing}.
	
	\begin{theorem}\label{t.turing}
		Let $G$ be a connected Lie group with a compatible left invariant metric $d$. If $G$ can be approximated by finite subgroups in the sense that for every $\eps>0$ there is a finite subgroup $F\leq G$ that is an $\eps$-net in $G$, then $G$ is compact and abelian.
	\end{theorem}
	
	As we have already mentioned, Turing formulates his theorem without the assumption of connectedness, but this condition \emph{cannot be omitted}.
	
	\section{Subgroups isolated up to conjugacy}
	
	\begin{defi}
		Let $K$ be a compact subgroup of a topological group $G$. Then the point $K\in\cals(G)$ is \emph{isolated up to conjugacy} if there exists an open neighborhood $\calu\subseteq\cals(G)$ of $K$ such that for every $H\in\calu$, there is some $g\in G$ such that $H=gKg^{-1}$.
	\end{defi}
	
	\begin{lemma}\label{l.Torus_semi_finite}
		If the identity component of a compact Lie group $K$ is a torus, then $K$ factors onto a semidirect product of a torus and a finite group.
	\end{lemma}
	\begin{proof}
		We shall use the well-known connection between the extensions of a group with an abelian kernel and the second cohomology of the group (see Brown \cite[Sec. IV.3]{Brown} for details.) By our assumption, there is a short exact sequence
		\[
		0\to T\to K\stackrel{\varphi}{\to} F\to 0,
		\]
		where $T=K^\circ$ is a torus, and $F$ is a finite group. The group 
		$K$ acts on $T$ by conjugation, and the action of $T$ on itself is trivial, so this actions factors to an action $\alpha\colon F\cong K/T\to \mathrm{Aut}(T)$ of $F$ on $T$, hence $T$ is a $K$-module. Following \cite{Brown}, we write the group operation on $T$ additively, and for $k\in K$ and $x\in T$, $\alpha(k)\big(x\big)$ will be denoted shortly by $kx$.
		
		The group $K$ is isomorphic to the semidirect product $T\rtimes_{\alpha}F$ if and only if there is a homomorphism $s\colon F\to K$ such that $\varphi\circ s=\textrm{id}_F$.
		Choosing an arbitrary map $s\colon F\to K$ such that $\varphi\circ s=\textrm{id}_F$ and $s(e_F)=e_K$, the failure of $s$ to be a homomorphism is measured by the $T$-valued normalized $2$-cocycle
		\[
		f_s\colon F\times F\to T,\qquad f_s(g,h)=s(g)s(h)s(gh)^{-1}.
		\]
		If $s'=as$, where $a\colon F\to T$ is an arbitrary map satisfying $a(e_F)=e_K=0_T$, then $f_{s'}$ differs from $f_s$ by the coboundary $\delta a(g,h)=a(g)+g a(h)-a(gh)$ of $a$, that is $f_{s'}(g,h)=f_s(g,h)+\delta a(g,h)$. Thus, the obstruction of $K$ to be isomorphic to the semidirect product $T\rtimes_{\alpha}F$ is the cohomology class $[f_s]\in H^2(F;T)$, not depending on $s$.
		
		Let $m=|F|$ be the cardinality of $F$. By \cite[Ch. III. Cor. 10.2]{Brown}, even if $[f_s]\neq 0$, we have $[mf_s]= 0$. This means that there is a map $a'\colon F\to T$ such that $a'(e_F)=0_T$ and $mf_s+\delta a'\equiv 0_T$. As $T$ is a divisible abelian group, there is a map $a\colon F\to T$ such that $a(e_F)=0_t$ and $ma \equiv a'$. Then setting $s'=as$, the cocycle $f_{s'}$ will take its values in the subgroup $T_m=\{x\in T\mid m x=0_T\}$. This implies that $s'(F)\cup T_m$ generates a finite subgroup $F'$ of $K$.
		
		It is clear that $TF'=K$ and $T\cap F'=T_m$. Factoring $K$ by $T_m$ we obtain a compact group with identity component equal to the torus $T/T_m$, hence $K/T_m$ is isomorphic to the semidirect product of $T/T_m$ and the finite group $F'/T_m\cong F$. 
	\end{proof}
	
	\begin{theorem}\label{t.isolated_characterized}
		Let $G$ be a Lie group with a compatible metric $d$. For a compact subgroup $K\leq G$ the following are equivalent:
		
		(1) The identity component $K^{\circ}$ of $K$ is not perfect, i.e., $K^\circ\neq (K^\circ)'$.
		
		(2) The subgroup $K$ has a quotient of the form
		$$\circle^m\rtimes F,$$
		where $m>0$,  $\circle^m$ is an $m$-dimensional torus and $F$ is a finite group.
		
		(3) There is a sequence of closed proper subgroups $K_n\lneqq K$ such that $K_n\overset{d_H}{\to}K$.
		
		(4) The point $K\in\cals(G)$ is not isolated up to conjugacy.
		
		(5) The conjugacy  set $\{gKg^{-1}\mid g\in G\}$ of $K$ is not open in $\cals(G)$.
	\end{theorem}
	
	\begin{remark}
		The proof elaborates on an answer of Nicolas Tholozan on MathOverflow \cite{Tholozan}.
	\end{remark}
	\begin{proof}  (4)$\iff$(5): It is immediate from Lemma \ref{l.right_multip_cont}.
		
		(1)$\implies$(2): First observe that $(K^\circ)'$ is a characteristic subgroup of $K^\circ$, therefore it is normal in $K$. The quotient group $K/(K^\circ)'$ is a compact group with identity component $K^\circ/(K^\circ)'$. Here $K^\circ/(K^\circ)'$ is a non-trivial connected compact abelian Lie group, hence a torus of dimension at least $1$, so we are done by Lemma \ref{l.Torus_semi_finite}.
		
		(2)$\implies$(3): Let $\eps>0$. Fix a quotient map $\varphi\colon K\to\circle^m\rtimes F$, and let $K_n\defeq\varphi^{-1}((C_n)^m\rtimes F)$, where $C_n$ is the cyclic group of order $n$. (It is easy to see that $(C_n)^m\rtimes F$ is a subgroup of $\circle^m\rtimes F$.) Then the subgroups $K_n$ are closed proper subgroups of $K$. Let $\bigcup_{i=1}^k U_i$ be a cover of $K$ with open sets of diameter $<\eps$. Since every surjective continuous homomorphism between compact groups is an open map, the system $\bigcup_{i=1}^k \varphi(U_i)$ is an open cover of $\circle^m\rtimes F$. Clearly, for sufficiently large $n$, the subgroup $(C_n)^m\rtimes F$ meets every $\varphi(U_i)$, hence $K_n$ meets every $U_i$. Then $K_n$ is an $\eps$-net in $K$ and we have $d_H(K_n,K)\leq\eps$.
		
		(3)$\implies$(4): It follows immediately from Lemma~\ref{l.nonconjugate}.
		
		(4)$\implies$(3): Let $(H_n)_{n\in\nat}$ be a sequence of compact subgroups of $G$ such that $H_n\hconv K$ and none of the $H_n$ is conjugate to $K$.
		
		By Theorem~\ref{t.mz} there are open sets $U_i$ with $K\subseteq U_i\subseteq G$ for every $i\in\nat$ such that for every subgroup $H$ of $G$ that lies in $U_i$ there exists $g_i\in B{\left(e_G,\frac{1}{2^i}\right)}$ such that $g_iHg_i^{-1}\leq K$. Let $(H_{n_i})$ be a subsequence of $H_n$ with $H_{n_i}\subseteq U_i$ for every $i\in\nat$. Let $g_i\in B{\left(e_G,\frac{1}{2^i}\right)}$ be such that $K_i\defeq g_iH_{n_i}g_i^{-1}\leq K$. 
		Then applying Lemma~\ref{l.right_multip_cont}, we get $K_i\hconv K$.
		
		(3)$\implies$(1): Let $(K_n)_{n\in\nat}$ be a sequence of closed proper subgroups of $K$ such that $K_n\hconv K$. Let $r\defeq\dim(K)$. Let $\frk_n$ be the Lie algebra of $K_n$ and let $\frk$ be the Lie algebra of $K$. We may assume that there is some $q\leq r$ such that $\dim(K_n)=q$ for each $n\in\nat$. If $q=r$, then it follows from Theorem~\ref{t.lie_group_fun_facts} (1) that $(K_n)^\circ=K^\circ$ for every $n$. Then each $K_n$ is a union of cosets of $K^\circ$. Since each of the groups $K_n$ is a proper subgroup of $K$, they cannot converge to $K$. Thus $q<r$.
		
		We can view the Lie algebras $\frk_n$ as elements of the Grassmannian manifold $Gr_q(\frk)$, which is compact and metrizable. 
		
		By compactness, we may assume that $\frk_n\to\frh$ for some $\frh\in Gr_q(\frk)$. 
		
		\textbf{Claim 1.} The linear subspace $\frh$ of $\frk$ is an ideal of $\frk$.
		
		For a $k\in K$ let $\Ad_k\colon \frk\to\frk$ be the adjoint operator of $k$ and let $\ad\colon \frk\to \End(\frk)$ be the adjoint representation of $\frk$.
		
		Let $k\in K$ and $x\in\frh$ be arbitrary. Pick $x_n\in\frk_n$ and $k_n\in K_n$ such that $x_n\to x$ in $\frk$ and $k_n\to k$ in $K$. By definition $\Ad_{k_n}(x_n)\in\frk_n$. Also $\Ad_{k_n}(x_n)\to\Ad_k(x)$ since $\Ad$ is continuous. It follows that $\Ad_k(x)\in\frh$.
		
		Now fix $x\in\frh$ and $y\in\frk$ arbitrarily. We need to prove that $[y,x]\in\frh$. We know that $[y,x]=\ad_y(x)$. Notice that
		$$\ad_y(x)=\tfrac{d}{dt}{\left(e^{\ad(ty)}(x)\right)}\big\vert_{t=0}$$
		and
		$$e^{\ad(ty)}(x)=\Ad_{\exp_K(ty)}(x),$$
		where $\exp_K\colon \frk\to K$ is the exponential map. As we have already verified, the curve $t\mapsto\Ad_{\exp_K(ty)}(x)$ maps into $\frh$, hence its derivative at $0$ is an element of $\frh$ as well. Therefore, we conclude $[y,x]\in\frh$, which proves Claim 1. \hfill $\blacksquare$
		
		By Theorem~\ref{t.lie_group_fun_facts} (3), there exists a unique connected, simply connected Lie group $\wtilde{H}$ with Lie algebra $\frh$, and the inclusion $\psi\colon \frh\embeds\frk$ gives rise to an immersion $\wtilde{\varphi}\colon\wtilde{H}\to K$ that is a group homomorphism as well. Then $\kernel\wtilde\varphi$ is a discrete Lie group and the quotient map $\wtilde H\to \wtilde H/\kernel\wtilde\varphi$ induces an embedding $\varphi\colon  \wtilde H/\kernel\wtilde\varphi\embeds 
		K$. Thus $H\defeq\varphi(\wtilde H/\kernel\wtilde\varphi)$ is a connected subgroup of $K$ with Lie algebra $\frh$. It follows from Theorem~\ref{t.normal_ideal} and Claim 1 that $H$ is a normal subgroup of $K^\circ$.
		
		We prove $(K^\circ)'\neq K^\circ$ by  seeking a contradiction from the assumption $(K^\circ)'= K^\circ$. By Theorem~\ref{t.lie_alg_decomposition} (2), if $K^\circ=(K^\circ)'$, then $\frk$ is a compact semisimple Lie algebra, therefore it is the direct sum $\fra_1\oplus\ldots\oplus\fra_k$ of compact simple Lie algebras $\fra_i$ and $\frh\nsub \frk$ is the direct sum of some of the ideals $\fra_i$. In particular, $\frh$ is compact and semisimple, implying that  both $\wtilde H$ and $H$ are compact Lie groups. 
		
		\textbf{Claim 2.} For the identity components of the groups $K_n$ we have $(K_n)^\circ\hconv H$.
		
		We are going to apply Theorem~\ref{t.exp_surj}. Fix an $\ad_{\frk}$-invariant inner product $\langle\, ,\rangle$ on $\frk$, and let $\varrho$ be the corresponding Euclidean metric on $\frk$. For a subalgebra $\frs\le\frk$, equip $\frs$ with the $\ad_{\frs}$-invariant inner product obtained as the restriction of $\langle\, ,\rangle$ onto $\frs\times\frs$. Choose for each $\frk_n$ a unit vector $x_n\in \frk_n$ such that  
		\[
		Ric_{\min}^{\frk_n}=\frac{1}{4}\min_{{x\in\frk_n\atop\|x\|=1}}\|\ad_{\frk_n}(x)\|_{HS}^2=\frac{1}{4}\|\ad_{\frk_n}(x_n)\|_{HS}^2.
		\]
		We must have $\liminf_{n\to\infty}Ric_{\min}^{\frk_n}>0$, otherwise we could find a subsequence $x_{n_1},x_{n_2},\dots$ converging to a unit vector $x\in \frh$ such that
		\[
		0=\lim_{i\to\infty} \frac{1}{4}\|\ad_{\frk_{n_i}}(x_{n_i})\|_{HS}^2=\frac{1}{4}\|\ad_{\frh}(x)\|_{HS}^2,
		\]
		which would contradict that $\frh$ has trivial center, therefore $\ad_{\frh}(x)\neq 0$.
		
		Thus, Theorem~\ref{t.exp_surj} gives us some $\Delta>0$ such that  $\exp_K(B_\varrho(0,\Delta)\cap\frh)=H^\circ=H$ and $\exp_K(B_\varrho(0,\Delta)\cap\frk_n)=(K_n)^\circ$ for all but a finite number of $n$. Since we have $\cl B_\varrho(0,\Delta)\cap\frk_n\to \cl B_\varrho(0,\Delta)\cap\frh$ (with respect to the Vietoris topology on $\calk(\frk)$) and $\exp_K$ is uniformly continuous on $\frk\cap \cl B_\varrho(0,\Delta)$, we conclude that $(K_n)^\circ\hconv H$, and Claim 2 is proved. \hfill$\blacksquare$
		
		As we have seen in the proof of (4)$\implies$(3), by replacing $(K_n)^\circ$ with a subsequence if necessary we can find elements $k_n\in K^\circ$ such that $\wtilde K_n\defeq k_n(K_n)^\circ k_n^{-1}\leq H$ and $\wtilde K_n\hconv~H$. Recall that $\dim(\wtilde K_n)=\dim(H)$. Since both $\wtilde K_n$ and $H$ are connected, it follows from Theorem~\ref{t.lie_group_fun_facts} (1) that $\wtilde K_n=H$ for every $n\in\nat$. Then $(K_n)^\circ=k_n^{-1}Hk_n=H$ since $H\nsub K^\circ$. Thus $H$ is normal in $K_n\cap K^\circ$ for every $n$.
		
		It is easy to see that $K_n\hconv K$ and $K_n\leq K$ implies $K_n\cap K^{\circ}\hconv K^{\circ}$. Therefore, we have $(K_n\cap K^{\circ})/H\to K^\circ/H$ with respect to the Vietoris topology on $K^\circ/H$. Here every quotient $(K_n\cap K^{\circ})/H$ is finite because $\dim(K_n\cap K^{\circ})=\dim(K_n)=\dim(H)$. Now by Theorem ~\ref{t.turing}, the quotient $K^{\circ}/H$ is abelian. Since $\dim(K)\neq\dim(H)$, it is also non-trivial. Thus $K^{\circ}\neq(K^{\circ})'$, a contradiction.
		
		The proof of Theorem~\ref{t.isolated_characterized} is complete.
	\end{proof}
	
	It is worth stating the special case of compact connected Lie groups on its own:
	
	\begin{cor}\label{c.connected}
		Let $G$ be a Lie group. A compact connected subgroup $K\in\cals(G)$ is isolated up to conjugacy if and only if $K=K'$. (Or equivalently if $K$ does not have a non-trivial abelian quotient.)
	\end{cor}
	
	We finish this section by a characterization of the minimal quasi-isomorphism classes of compact Lie groups that are not isolated up to conjugacy in Lie groups, where minimality is taken with respect to the quasi-order $\succeq$ defined in the introduction. First we note that isomorphic compact Lie groups are obviously quasi-isomorphic, but the converse is false.
	\begin{example}
		Identify $\circle^2$ with $\mathbb R^2/\mathbb Z^2$ and consider the actions $\alpha$ and $\beta$ of $\mathbb Z_2=\{\pm 1\}$ on $\circle^2$ given by the involutions
		\[
		\alpha_{-1}\colon[(x,y)]\mapsto[(x,-y)]\text{ and }\beta_{-1}\colon[(x,y)]\mapsto[(x+y,-y)]. 
		\]
		The Lie groups $G_{\alpha}=\circle^2\rtimes_{\alpha}\mathbb Z_2$ and $G_{\beta}=\circle^2\rtimes_{\beta}\mathbb Z_2$ are not isomorphic, as the center of $G_{\alpha}$ is a union of two disjoint circles, while the center of $G_{\beta}$ is just a circle. However, $G_{\alpha}$ and $G_{\beta}$ are quasi-isomorphic, as the quotients of $G_{\alpha}$ by the normal subgroup $\{(k+l/2,l/2)\mid (k,l)\in \mathbb Z^2\}/\mathbb Z^2$ is isomorphic to $G_{\beta}$, while the quotient of $G_{\beta}$ by the normal subgroup $(\frac{1}{2}\mathbb Z\oplus \mathbb Z)/\mathbb Z^2$ is isomorphic to $G_{\alpha}$.
	\end{example}
	\begin{prop}\label{prop:minimal_groups}
		The $\succeq$-minimal elements  among the quasi-isomorphism classes of compact Lie groups not isolated up to conjugacy in Lie groups are exactly those that can be represented by groups of the form $\circle^m\rtimes_{\alpha} F$, where $m>0$, and $F$ is a finite group with a faithful and rationally irreducible conjugacy representation $\alpha\colon F\to\textrm{GL}(m,\mathbb Z)\cong \textrm{Aut}(\circle^m)$. 
	\end{prop}
	\begin{proof}
		By Theorem \ref{t.isolated_characterized}, a minimal quasi-isomorphism class must contain a group of the form $G=\circle^m\rtimes_{\alpha} F$, where $m>0$, and $F$ is a finite group. If $\alpha$ is not faithful, then $\ker \alpha \leq F\leq G$ is a normal subgroup of $G$ and $G/(\ker \alpha)\preceq G$ is not isolated up to conjugacy as it is the semidirect product of $\circle^m$ and a finite group. On the other hand, $G/(\ker \alpha)\not\succeq G$ since $G/(\ker \alpha)$ has fewer connected components than $G$, therefore the quasi-isomorphism class of $G$ is not minimal. Similarly, if the representation $\alpha\colon F\to GL(m,\mathbb Z)$ has a nontrivial invariant $\mathbb Q$-linear subspace $0<V<\mathbb Q^m$, then the torus $N=\mathbb R V/(V\cap \mathbb Z^m)$ is a closed normal subgroup of $G$ and the quotient group $G/N$ remains not isolated up to conjugacy, while its quasi-isomorphism class is strictly smaller than that of $G$. We conclude that the conditions on $\alpha$ are necessary for the minimality.
		
		To show sufficiency, assume that $\alpha$ is a faithful and rationally irreducible representation of $F$. We have to show that for any closed normal subgroup  $N\triangleleft G$, the quotient group $G/N$ is either isolated up to conjugacy, or quasi-isomorphic to $G$. Consider first the closed normal subgroup $(N\cap \circle^m)^{\circ}\triangleleft G$. Its Lie algebra is an $\alpha(F)$-invariant linear subspace of $\mathbb R^m$, which has a basis consisting of some vectors with integer coordinates, hence, by the rational irreducibility of $\alpha$, the subgroup $(N\cap \circle^m)^{\circ}$ is either the torus $\circle^m$ or the trivial group. If $(N\cap \circle^m)^{\circ}=\circle^m$, then $G/N$ is finite and consequently isolated up to conjugacy. Now assume that $(N\cap \circle^m)^{\circ}$ is the trivial group, i.e., $N$ is finite. For any $n\in N$, the set $\{tnt^{-1}n^{-1}\mid t\in \circle^m\}$ is connected, and contained in the discrete group $N$, hence it has only one element, the identity element $e_G=e_Gne_G^{-1}n^{-1}$ of $G$. Thus, $N$ commutes with the elements of $\circle^m$, consequently, the image of $N$ under the quotient map $G\to F$ is in the kernel of $\alpha$. As $\alpha$ is faithful, this yields $N<\circle^m$. If $k=|N|$, then $N$ is contained in the characteristic subgroup $\Tilde N=(\frac{1}{k}\mathbb Z)^m/\mathbb Z^m<\circle^m$, therefore $ G/\Tilde N \cong G$ is a quotient of the group $G/N$, showing that $G/N$ is quasi-isomorphic to $G$.   
	\end{proof}
	\section{A sufficient condition for the openness of \texorpdfstring{$\cals(f)$}{S(f)}}\label{s.an_open_map}
	
	As we have mentioned in the introduction, the following question originates from Proposition~\ref{p.continuous_image}.
	
	\begin{q:open}
		For which compact groups $G$ and $H$ is it true that for any surjective continuous homomorphism $f\colon G\to H$, the map $\cals(f)\colon\cals(G)\to\cals(H)$ is surjective and open?
	\end{q:open}
	
	We present a partial answer to Question~\ref{q.open}.
	
	\begin{theorem}\label{t.open_map}
		Let $G$ be a compact group and let $H$ be a compact Lie group. Let $f\colon G\to H$ be a surjective continuous homomorphism. Then the map $\cals(f)\colon \cals(G)\to\cals(H)$ is also surjective, continuous and open.
	\end{theorem}
	
	\begin{proof}
		\textbf{Surjectivity.} Clearly, for any $L\in\cals(H)$ the inverse image $f^{-1}(L)$ is a compact subgroup of $G$ and $\cals(f)\big(f^{-1}(L)\big)=L$.
		
		\textbf{Continuity} follows immediately from Lemma~ \ref{lem:_folytonossag}.
		
		\textbf{Openness.} It suffices to prove that for any $K\in\cals(G)$ and any open neighborhood $\Tilde{\calu}$ of $K$, the set $\cals(f)\big(\Tilde{\calu}\big)$ is a neighborhood of $f(K)$.
		
		Fix $K\in\cals(G)$ and a neighborhood $\Tilde{\calu}$ of $K$ arbitrarily. Let $L\defeq f(K)$. By Lemma~\ref{l.right_multip_cont} there exists a neighborhood $B$ of $e_G$ and an open neighborhood $\calu$ of $K$ such that $g^{-1}\calu g \subseteq \Tilde{\calu}$ for all $g \in B$. We may assume that $\calu$ is of the form $\calu=\calv(U_0,\dots,U_n)$, where  $U_0, U_1,\ldots,U_n$ are open sets in $G$ such that $K\subseteq U_0$ and $K\cap U_1\neq\emptyset,\ldots,\allowbreak K\cap U_n\neq\emptyset$.
		
		Since $f|_K\colon K\to L$ is an open map by Proposition \ref{p.quotient_map}, the sets $f|_K(U_1\cap K),\ldots,\allowbreak f|_K(U_n\cap K)$ are relative open in $L$. Hence let $V_1,\dots,V_n$ be open sets in $H$ with $V_i\cap L=f|_K(U_i\cap K)$. Clearly, $\calw\defeq \calv\big(H,V_1,\dots,V_n\big)$ is an open neighborhood of $L$ in $\cals(H)$. Let $\calc\subseteq\calw$ be a compact neighborhood of $L$ in $\cals(H)$.
		By Lemma~\ref{l.right_multip_cont}, there is an open neighborhood $O$ of the identity of $H$ such that for every $h\in O$ and $N\in\calc$ we have $hNh^{-1}\in\calw$. We may also assume $O\subseteq f(B)$ because $f$ is an open homomorphism.
		
		Now by Theorem~\ref{t.mz} there is some open set $W_O\supseteq L$ in $H$ such that for every subgroup $N$ of $H$ that lies in $W_O$ there exists $h\in O$ such that $hNh^{-1}\leq L$. Let $\Tilde\calw \subseteq\calc$ be a neighborhood of $L$ in $\cals(H)$ such that for every $N\in\Tilde\calw$ we have $N\subseteq W_O$. Now it suffices to prove the following claim.
		
		\textbf{Claim.} We have $\Tilde\calw\subseteq\cals(f)(\Tilde{\calu)}$.
		
		Pick any $N\in\Tilde\calw$. By $N\subseteq W_O$ and the definition of $W_O$, there is some $h\in O$ such that $hNh^{-1}\leq L$. Also by $N\in\calc\subseteq\calw$ and the choice of $O$, we have $hNh^{-1}\in\calw$. That is, we can pick points $y_i\in (hNh^{-1}\cap V_i)\subseteq L\cap V_i$ for each $1\leq i\leq n$.
		
		Now consider the compact subgroup $M\defeq(f|_{K})^{-1}(hNh^{-1})$ of $K$. By the definition of $V_i$, there is $x_i\in U_i\cap M$ with $f|_K(x_i)=y_i$ for each $1\leq i\leq n$. Since also $M\subseteq K\subseteq U_0$, we have $M\in\calu$. Let $g\in B$ with $f(g)=h$. Now we have $g^{-1}Mg\in \Tilde{\calu}$ and $\cals(f)(g^{-1}Mg)=f(g^{-1}Mg)=f(g)^{-1}f(M)f(g)=h^{-1}(hNh^{-1})h=N$, which proves the claim.
	\end{proof}
	
	\section{Acknowledgements}
	
	The first author was supported by the National Research Development and Innovation Office (NKFIH) grant No K-128862. During the research he also enjoyed the hospitality of the Alfr\'ed R\'enyi Institute of Mathematics as a guest researcher. 
	
	The project leading to this application has received funding from the European Research Council (ERC) under the European Union’s Horizon 2020 research and innovation programme (grant agreement No 741420).
	
	The second author was supported by the ÚNKP-21-3 New National Excellence Program of the Ministry for Innovation and Technology from the source of the National Research, Development and Innovation Fund.
	
	The third author was supported by the ÚNKP-21-1 New National Excellence Program of the Ministry for Innovation and Technology from the source of the National Research, Development and Innovation Fund.
	
	\bibliographystyle{IEEEtran}
	\bibliography{biblio}
	
\end{document}